\documentclass{article}
\usepackage{amsmath,amssymb,amsfonts}
\usepackage{theorem}
\usepackage{color}
\usepackage{hyperref}
\usepackage{bbm}

\theoremstyle{change}
\newtheorem{definition}{Definition:}[section]
\newtheorem{proposition}[definition]{Proposition:}
\newtheorem{theorem}[definition]{Theorem:}
\newtheorem{lemma}[definition]{Lemma:}
\newtheorem{corollary}[definition]{Corollary:}
{\theorembodyfont{\rmfamily}
\newtheorem{remark}[definition]{Remark:}
}
{\theorembodyfont{\rmfamily}

}
\newenvironment{proof}
  {{\bf Proof:}}
  {\qquad \hspace*{\fill} $\Box$}

\begin{document}

\title{Control sets of linear systems on Lie groups}
\author{Adriano Da Silva\thanks{%
Supported by Fapesp grant $n^{o}$ 2013/19756-8} \\
Instituto de Matem\'{a}tica\\
Universidade Estadual de Campinas\\
Cx. Postal 6065, 13.081-970 Campinas-SP, Brasil \and V\'{\i}ctor Ayala%
\thanks{%
Supported by Proyecto Fondecyt n$%
{{}^\circ}%
$ 1150292. Conicyt, Chile.} \\
Instituto de Alta Investigaci\'{o}n\\
Universidad de Tarapac\'{a}\\
Casilla 7D, Arica, Chile\\
and \and Guilherme Zsigmond \\
Universidad Cat\'{o}lica del Norte\\
Av. Angamos 0610, Antofagasta, Chile}
\date{\today }
\maketitle

\begin{abstract}
Like in the classical linear Euclidean system, we want to characterize for a
linear system on a connected Lie group $G$ its control set with nonempty
interior $\mathcal{C}$ that contains the identity of $G$. We show that many
topological properties of $\mathcal{C}$ are intrinsically connected with the
eigenvalues of the derivation $\mathcal{D}$ associated to the drift $%
\mathcal{X}$. In particular, we prove that if $G$ is a solvable Lie group or
if the Lie subgroup $G^{0}=\langle \exp (\mathfrak{g}^{0})\rangle $ is
compact, where $\mathfrak{g}^{0}=\bigoplus_{\alpha }\{ \mathfrak{g}_{\alpha
};\, \mathrm{Re}(\alpha )=0\}$, then $\mathcal{C}$ is the only control set
with nonempty interior of the linear system. Here $\mathfrak{g}_{\alpha }$
stands for the generalized eigenspace associated to the eigenvalue $\alpha $%
. Furthermore, for nilpotent Lie groups we characterize when the control set 
$\mathcal{C}$ is bounded. In particular, if $G$ is a nilpotent simply
connected Lie group, a linear system on $G$ admits a bounded control set
only if $\mathcal{D}$ is hyperbolic.
\end{abstract}

\bigskip



\section{Introduction}

Throughout the paper $G$ stands for a connected Lie group with Lie algebra $%
\mathfrak{g}$ and dimension $d$. In \cite{VAJT} the authors introduced the
class of linear system on $G$, determined by the family of differential
equations%
\begin{equation}
\Sigma :\dot{g}(t)=\mathcal{X}(g(t))+\sum_{j\text{ }=1}^{m}u_{j}(t)\text{ }%
X^{j}(g(t)),  \label{Linear system}
\end{equation}%
here the drift $\mathcal{X}$ with flow $(\varphi _{t})_{t\in \mathbb{R}}$ is
called a linear vector field which is associated to a $\mathfrak{g}$%
-derivation $\mathcal{D}$ through the formula 
\begin{equation*}
(d\varphi _{t})_{e}=\mathrm{e}^{t\mathcal{D}}\; \; \; \mbox{ for all }\; \;
\;t\in \mathbb{R}.
\end{equation*}

The vector fields $X^{j}$ are right invariant and $u\in \mathcal{U}\subset
L^{\infty }(\mathbb{R},\Omega \subset \mathbb{R}^{m})$ is the class of
admissible controls where $\Omega \subset \mathbb{R}^{m}$ is a compact,
convex subset with $0\in \mathrm{int}\Omega $.

In \cite{JPh1} the author shows that this class of system is also important
in applications. Actually, it is shown there that an arbitrary affine
control systems on a manifold $M$ which dynamic generate a finite Lie
algebra is equivalent to an invariant control systems or a linear control
system on a homogeneous space.

On the other hand, it is very well known that the classical linear system on
the Euclidean space $\mathbb{R}^{d}$ is one of the most relevant control
systems and it can be written as 
\begin{equation*}
\dot{x}(t)=Ax(t)+\sum_{j\text{ }=1}^{m}u_{j}(t)\text{ }b^{j},\hspace{0.5cm}%
b^{j}\in \mathbb{R}^{d}\text{ and }u\in \mathcal{U}.
\end{equation*}%
Here $A\in \mathrm{gl}(d,\mathbb{R}),$ the Lie algebra of the real matrices
of order $d.$ Since $\mathbb{R}^{d}$ is commutative, any constant vector $%
b^{j}$ is an invariant vector field$.$ Just observe that $e^{tA}\in GL^{+}(d,%
\mathbb{R)}$ $=Aut(\mathbb{R}^{d})$. This fact gave the first insight in
order to generalize the notion of linear systems from $\mathbb{R}^{d}$ to $%
G. $ In fact, the flow $(\varphi _{t})_{t\in \mathbb{R}}$ of $\mathcal{X}$
is a $1$-parameter group of $G$-automorphisms, and $(\frac{d}{dt}%
)_{t=0}(\varphi _{t})$ coincides with the derivation $\mathcal{D}$\
associated to $\mathcal{X},$ like in the Euclidean case$.$ Of course the $%
b^{j}$ bi-invariant vector fields on $\mathbb{R}^{d}$ are the columns of the
classical cost matrix $B$.

It was shown in \cite{DaSilva1} that controllability of the linear systems $%
\Sigma $ on a Lie group is really an exceptional property. In fact, assume
that $G$ is nilpotent and the accessibility set $\mathcal{A}$ from the
identity element $e$ of $G$ is an open set. It turns out that 
\begin{equation*}
\Sigma \text{ is controllable on }G\Leftrightarrow Spec(\mathcal{D})\cap 
\mathbb{R}=\left \{ 0\right \} .
\end{equation*}%
Furthermore, recently the authors in \cite{DaAy} proved that $Spec(\mathcal{D%
})\cap \mathbb{R}=\left \{ 0\right \} $ implies controllability for any Lie
group with finite semisimple center, that is, any Lie group that admits a
maximal semisimple Lie subgroup with finite center.

To understand the controllability behavior of linear systems on Lie groups,
we need to approach the problem in a more realistic way. We turn our
attention to the maximal subsets of $G$ where controllability of the system
holds, means, the control sets.

Like in the classical linear system in this paper we characterize the
control set with nonemtpy interior of (\ref{Linear system}) around the
identity. As we were expected, many topological properties of $\mathcal{C}$
are intrinsically connected with the eigenvalues of the derivation $\mathcal{%
D}$ associated to the drift $\mathcal{X}$.

In particular, we prove that if $G$ is a solvable Lie group or if $G^{0}$ is
a compact subgroup 
\begin{equation*}
\mathcal{C}=\mathrm{cl}(\mathcal{A})\cap \mathcal{A}^{\ast }
\end{equation*}%
is the only control set of the linear system $\Sigma $. Here, $\mathcal{A}$
and $\mathcal{A}^{\ast }$ denote the reachable (controllable) set of the
identity element $e$ of $G$ respectively. Furthermore, 
\begin{equation*}
\mathcal{C}\text{ is bounded}\Rightarrow G^{0},\mathrm{cl}(\mathcal{A}%
_{G^{-}})\mbox{ and }\mathrm{cl}(\mathcal{A}_{G^{+}}^{\ast })\text{ are
compact sets},
\end{equation*}%
where $\mathcal{A}_{G^{-}}$ ($\mathcal{A}_{G^{+}}^{\ast }$) is the
intersection of $\mathcal{A}$ ($\mathcal{A}^{\ast }$) with the stable
(unstable) manifold of $\varphi $. For the nilpotent case the converse is
true. In particular, if $G$ is a nilpotent simply connected Lie group we
show that 
\begin{equation*}
\mathcal{C}\text{ is bounded }\Leftrightarrow \mathrm{cl}(\mathcal{A}%
_{G^{-}})\text{ and }\mathrm{cl}(\mathcal{A}_{G^{+}}^{\ast })\text{ are
compacts and }\mathcal{D}\text{ is hyperbolic.}
\end{equation*}

The paper is structured as follows: Section 2 introduces the notion of
affine control systems and its control sets on an arbitrary differentiable
manifold. We state here basic properties of control sets with nonempty
interior. Furthermore, we define linear vector field and linear systems.
Associated with any $\mathfrak{g}$-derivation there are several Lie algebras
and their corresponding Lie groups, connected with the reachable and
controllable sets of the system. We take care of this decomposition here. In
Section 3 we analyze the control sets of the linear system (\ref{Linear
system}). By a general result from \cite{FCWK}, it follows that around the
identity of $G$ there exists one of these possible sets. Then we focus our
attention on its properties. In particular, we establish necessary and
sufficient conditions to decide whenever this set is invariant. On the other
hand, we analyze under which circumstances the control set is the whole $G$
and when it is bounded. At the end of Section 3 we prove our main result. We
show that the control set around the identity is the only maximal set of
approximate controllability for Lie groups where we can decompose $G$ as the
product of the subgroups associated with the real parts of the eigenvalues
of $\mathcal{D}$.

\section{Preliminaries}

In this Section we state the definitions and main results concerning to
affine control system, control sets, linear vector field the associated
subalgebras and the corresponding subgroups. For more on the subjects the
reader could consult \cite{VASM}, \cite{VAJT}, \cite{FCWK}, \cite{DaSilva1}, 
\cite{JPh}, \cite{JPhDM} and \cite{JPh1}.

\subsection{Affine control systems and its control sets}

Let $M$ be a $d$ dimensional smooth manifold. By an \emph{affine control
system} in $M$ we understand the family of ordinary differential equations 
\begin{equation}
\dot{x}(t)=f^{0}(x(t))+\sum_{j\text{ }=1}^{m}u_{j}(t)\text{ }f^{j}(x(t)),\;
\; \; \;u=(u_{1},\ldots ,u_{m})\in \mathcal{U}  \label{controlsystem}
\end{equation}%
where $f^{j}\in X^{\infty }(M)$, $j=0,1,\ldots ,m$ are arbitrary $\mathcal{C}%
^{\infty }$ vector fields on $M.$

As usual, the family $\mathcal{U}\subset L^{\infty }(\mathbb{R},\Omega
\subset \mathbb{R}^{m})$ is the class of restricted admissible control
functions where $\Omega \subset \mathbb{R}^{m}$ with $0\in \mathrm{int}%
\Omega ,$ is a compact and convex set called the $control$\emph{\ }$range$
of the system.

Since our work focus on Lie groups we assume that the vector fields $%
f^{j},\,j=0,1,\ldots ,m$ are analytical. Therefore, for each control
function $u\in \mathcal{U}$ and each initial value $x\in M$ there exists an
unique solution $\phi (t,x,u)$ defined on an open interval containing $t=0,$
satisfying $\phi (0,x,u)=x$. Note that in general $\phi (t,x,u)$ is only a
solution in the sense of Carath\'{e}odory, i.e., a locally absolutely
continuous curve satisfying the corresponding differential equation almost
everywhere. Without lost of generality we assume that any solution can be
extended to the whole real line. Hence, we obtain a mapping 
\begin{equation*}
\phi :\mathbb{R}\times M\times \mathcal{U}\rightarrow M,\; \;
\;(t,x,u)\mapsto \phi (t,x,u),
\end{equation*}%
satisfying the \emph{cocycle property} 
\begin{equation*}
\phi (t+s,x,u)=\phi (t,\phi (s,x,u),\Theta _{s}u)
\end{equation*}%
for all $t,s\in \mathbb{R}$, $x\in M$, $u\in \mathcal{U}.$ Here for any $%
t\in \mathbb{R}$ the map $\Theta _{t}$ is the \emph{shift flow} on $\mathcal{%
U}$ defined by 
\begin{equation*}
(\Theta _{t}\text{ }u)(s):=u(t+s).
\end{equation*}%
Instead of $\phi (t,x,u)$ we usually write $\phi _{t,u}(x)$. Note that
smoothness of the vector fields $f^{0},$ $f^{1},\ldots ,f^{m}$ implies the
smoothness of $\phi _{t,u}$. Moreover, it follows directly from the cocycle
property that

\begin{itemize}
\item[$i)$] the inverse of the diffeomorphism $\phi _{t,u}$ exists and it is
given by $\phi _{-t,\Theta _{t}u}$

\item[$ii)$] the fact that for any $t>0,\phi _{t,u}(g)$ depends just on $%
u|_{[0,t]}$ implies that 
\begin{equation*}
\phi _{t,u_{1}}(\phi _{s,u_{2}}(g))=\phi _{t+s,u}(g)
\end{equation*}%
where $u\in \mathcal{U}$ is defined through concatenation by 
\begin{equation*}
u(\tau )=\left \{ 
\begin{array}{c}
u_{1}(\tau )\text{ for }\tau \in \lbrack 0,s] \\ 
u_{2}(\tau -s)\text{ for }\tau \in \lbrack s,t+s].%
\end{array}%
\right.
\end{equation*}%
For any $x\in M$ and $\tau >0$ the sets 
\begin{equation}
\begin{array}{l}
\mathcal{A}_{\leq \text{ }\tau }(x):=\{y\in M:\exists u\in \mathcal{U},\text{
}t\in \lbrack 0,\tau ]\text{ with}\; \;y=\phi _{t,u}(x)\} \; \; \\ 
\\ 
\mathcal{A}_{\tau }(x):=\{y\in M:\exists u\in \mathcal{U},\; \;y=\phi _{\tau
,u}(x)\} \\ 
\\ 
\mathcal{A}(x):=\bigcup_{\tau \text{ }>\text{ }0}\mathcal{A}_{\leq \text{ }%
\tau }(x),%
\end{array}
\label{reachablesets}
\end{equation}%
are the \emph{set of reachable point from }$x$\emph{\ up to time }$\tau $,
the \emph{set of reachable points from }$x$\emph{\ at time }$\tau $ and the 
\emph{reachable set of}\textbf{\ }$x$, respectively. In the same way, for
any $\; \tau >0$ the sets 
\begin{equation}
\begin{array}{l}
\mathcal{A}_{\leq \text{ }\tau }^{\ast }(x):=\{y\in M:\exists u\in \mathcal{U%
},t\in \lbrack 0,\tau ];\; \; \phi _{t,u}(y)=x\} \\ 
\\ 
\mathcal{A}_{\tau }^{\ast }(x):=\{y\in M:\exists u\in \mathcal{U},\;
\;y=\phi _{\tau ,u}(x)\} \\ 
\\ 
\mathcal{A}^{\ast }(x):=\bigcup_{\tau \text{ }>\text{ }0}\mathcal{A}_{\leq
\tau }^{\ast }(x),%
\end{array}
\label{controllablesets}
\end{equation}%
are called the \emph{set of controllable points to }$x$\emph{\ within time }$%
\tau $,\emph{\ the set of controllable points to }$x$\emph{\ in time }$\tau $%
\emph{\ and the controllable set of }$x$, respectively.
\end{itemize}

\begin{definition}
We say that the system (\ref{controlsystem})
\end{definition}

\begin{itemize}
\item[$i)$] is \emph{locally accessible at} $x$ if for all $\tau >0$ the
sets $\mathcal{A}_{\leq \text{ }\tau }(x)$ and $\mathcal{A}_{\leq \text{ }%
\tau }^{\ast }(x)$ have nonempty interior

\item[$ii)$] is \emph{locally accessible} if it is locally accessible at
every $x\in M$

\item[$iii)$] satisfy the \emph{Lie algebra rank condition} (LARC) if $%
\mathcal{L}(x)=T_{x}M$ for any $x\in M$, where 
\begin{equation*}
\mathcal{L}\text{ is the smallest }\mathfrak{g}\text{-subalgebra containing
any }f^{j}\text{, }j=0,1,\ldots ,m.
\end{equation*}%
Moreover, the system is locally accessible at $x\in M$ if it satisfies the
Lie algebra rank condition at the point $x$.
\end{itemize}

On the other hand, a more convenient approach to understand the
controllability behavior of a linear system comes from the notion of control
set.

\begin{definition}
A nonempty set $\mathcal{C}\subset M$ is called a \emph{control set} of the
affine system (\ref{controlsystem}) if it is
\end{definition}

\begin{itemize}
\item[$i)$] controlled invariant, that is, for every $x\in M$ there exists $%
u\in \mathcal{U}$ such that $\phi (\mathbb{R},x,u)\subset \mathcal{C}$

\item[$ii)$] approximate controllable, that is, $\mathcal{C}\subset \mathrm{%
cl}(\mathcal{A}(x))$ for every $x\in \mathcal{C}$

\item[$iii)$] is maximal with properties $(i)$ and $(ii)$
\end{itemize}

By Proposition 3.2.4 of \cite{FCWK}, a set $\mathcal{C}$ that is maximal
with respect to the property $(ii)$ above and satisfies $\mathrm{int}\,%
\mathcal{C}\neq \emptyset $ is a control set.

A control set $\mathcal{C}$ is said to be \textit{invariant in positive time}
if $\phi_{t, u}(\mathcal{C})\subset \mathcal{C}$ for any $t>0$ and $u\in%
\mathcal{U}$. In the same way, we say that $\mathcal{C}$ is \textit{%
invariant in negative time} if $\phi_{-t, u}(\mathcal{C})\subset \mathcal{C}$
for any $t>0$ and $u\in \mathcal{U}$.

The next proposition summarizes the main properties of control sets. Its
proof can be mainly found in \cite{FCWK} Theorem 3.1.5.

\begin{proposition}
\label{controlsets} Assume that the system (\ref{controlsystem}) is locally
accessible and let $\mathcal{C}$ be a control set with nonempty interior. It
holds
\end{proposition}

\begin{itemize}
\item[1.] $\mathcal{C}$ is connected and $\mathrm{cl}(\mathrm{int}\, 
\mathcal{C})=\mathrm{cl}(\mathcal{C})$

\item[2.] $\mathrm{int}\, \mathcal{C}\subset \mathcal{A}(x)$ for any $x\in 
\mathcal{C}.$ For any $y\in \mathrm{int}\, \mathcal{C}$ 
\begin{equation}
\mathcal{C}=\mathrm{cl}(\mathcal{A}(y))\cap \mathcal{A}^{\ast }(y).
\label{control}
\end{equation}%
In particular the system is controllable on $\mathrm{int}\, \mathcal{C}$.

\item[3.] Assume that $\phi _{t,u}(x)$ is a periodic trajectory, that is, $%
\phi _{t+s,u}(x)=\phi _{t,u}(x)$ for some $s>0$ and all $t\in \mathbb{R}$.
Therefore, if $x\in \mathrm{int}\, \mathcal{C}$ then $\phi _{t,u}(x)\in 
\mathrm{int}\, \mathcal{C}$ for all $t\in \mathbb{R}$.

\item[4.] $\mathcal{C}$ is closed $\Leftrightarrow \mathcal{C}$ is invariant
in positive time $\Leftrightarrow \mathcal{C}=\mathrm{cl}(\mathcal{A}(g))$
for any $g\in \mathcal{C}$;

\item[5.] $\mathcal{C}$ is open $\Leftrightarrow \mathcal{C}$ is invariant
in negative time $\Leftrightarrow \mathcal{C}=\mathcal{A}^*(g)$ for any $g\in%
\mathcal{C}$;
\end{itemize}

\subsection{Linear vector fields and decompositions}

Let $G$ be a connected Lie group with Lie algebra $\mathfrak{g}$.

\begin{definition}
A vector field $\mathcal{X}$ on $G$ is said to be \emph{linear} if its flows 
$(\varphi _{t})_{t\text{ }\in \text{ }\mathbb{R}}$ is a $1$-parameter group
of $G$-automorphisms.
\end{definition}

Certainly, the vector field $\mathcal{X}$ is complete. Furthermore, one can
associate to $\mathcal{X}$ a derivation $\mathcal{D}$ of $\mathfrak{g}$
defined by 
\begin{equation*}
\mathcal{D}Y=-[\mathcal{X},Y](e),\mbox{ for all }Y\in \mathfrak{g}.
\end{equation*}%
The relation between $\varphi _{t}$ and $\mathcal{D}$ is given by the
formula 
\begin{equation}
(d\varphi _{t})_{e}=\mathrm{e}^{t\mathcal{D}}\; \; \; \mbox{ for all }\; \;
\;t\in \mathbb{R}  \label{derivativeonorigin}
\end{equation}%
which implies that 
\begin{equation*}
\varphi _{t}(\exp Y)=\exp (\mathrm{e}^{t\mathcal{D}}Y),\mbox{ for all }t\in 
\mathbb{R},Y\in \mathfrak{g}.
\end{equation*}

On the other hand, if the group is simply connected any derivation $\mathcal{%
D}$ has an associated linear vector field $\mathcal{X=X}^{D}$ through the
same formula above. For connected Lie groups the same is true when $\mathcal{%
D\in }$ $\mathfrak{aut(}G\mathfrak{)}$ the Lie algebra of $Aut(G)$ the Lie
group of $G$-automorphism (see \cite{VAJT}).

Next we explicitly some decomposition of $\mathfrak{g}$ and $G$ induced by
the derivation $\mathcal{D}$. Let us consider the generalized eigenspaces of 
$\mathcal{D}$ given by 
\begin{equation*}
\mathfrak{g}_{\alpha }=\{X\in \mathfrak{g}:(\mathcal{D}-\alpha )^{n}X=0%
\mbox{
for some }n\geq 1\}
\end{equation*}%
where $\alpha $ is an eigenvalue of $\mathcal{D}$.

It turns out that $[\mathfrak{g}_{\alpha },\mathfrak{g}_{\beta }]\subset 
\mathfrak{g}_{\alpha +\beta }$ when $\alpha +\beta $ is an eigenvalue of $%
\mathcal{D}$ and zero otherwise. This fact allow to us to decompose $%
\mathfrak{g}$ as 
\begin{equation*}
\mathfrak{g}=\mathfrak{g}^{+}\oplus \mathfrak{g}^{0}\oplus \mathfrak{g}^{-}
\end{equation*}%
where 
\begin{equation*}
\mathfrak{g}^{+}=\bigoplus_{\alpha :\, \mathrm{Re}(\alpha )\text{ }>\text{ }%
0}\mathfrak{g}_{\alpha },\hspace{1cm}\mathfrak{g}^{0}=\bigoplus_{\alpha :\,%
\mathrm{Re}(\alpha )\text{ }=\text{ }0}\mathfrak{g}_{\alpha }\hspace{1cm}%
\mbox{ and }\hspace{1cm}\mathfrak{g}^{-}=\bigoplus_{\alpha :\, \mathrm{Re}%
(\alpha )\text{ }<\text{ }0}\mathfrak{g}_{\alpha }.
\end{equation*}%
It is easy to see that $\mathfrak{g}^{+},\mathfrak{g}^{0},\mathfrak{g}^{-}$
are Lie algebras and $\mathfrak{g}^{+}$, $\mathfrak{g}^{-}$ are nilpotent
(see Proposition 3.1 of \cite{SM1}).

We denote by $G^{+}$, $G^{-}$, $G^{0}$, $G^{+,0},$ and $G^{-,0}$ the
connected Lie subgroups of $G$ with Lie algebras $\mathfrak{g}^{+}$, $%
\mathfrak{g}^{-}$, $\mathfrak{g}^{0}$, $\mathfrak{g}^{+,0}:=\mathfrak{g}%
^{+}\oplus \mathfrak{g}^{0}$ and $\mathfrak{g}^{-,0}:=\mathfrak{g}^{-}\oplus 
\mathfrak{g}^{0}$ respectively.

\subsection{Linear systems on Lie groups}

A linear system on a Lie group $G$ is a control-affine system 
\begin{equation}
\dot{g}(t)=\mathcal{X}(g(t))+\sum_{j=1}^{m}u_{j}(t)\text{ }X^{j}(g(t)),
\label{linearsystem}
\end{equation}%
where the drift vector field $\mathcal{X}$ is a linear vector field, $X^{j}$
are right invariant vector fields and $u=(u_{1},\cdots ,u_{m})\in \mathcal{U}
$ as before.

For a given $g\in G$, $u\in \mathcal{U}$ and $t\in \mathbb{R}$ the solution
of the linear system (\ref{linearsystem}) starting at $g$ reads as 
\begin{equation}
\phi _{t,u}(g)=L_{\phi _{t,u}}(\varphi _{t}(g))=\phi _{t,u}\, \varphi _{t}(g),
\label{solutionform}
\end{equation}%
where $\phi _{t,u}=\phi _{t,u}(e)$ is the solution of (\ref{linearsystem})
starting at the identity element $e\in G$ (see for instance \cite{DaSilva}).

Let us denote by $\mathcal{A}_{\leq \text{ }\tau }$, $\mathcal{A}_{\tau }$
and $\mathcal{A}$ the sets $\mathcal{A}_{\leq \text{ }\tau }(e)$, $\mathcal{A%
}_{\tau }(e)$ and $\mathcal{A}(e)$, respectively. For any $u\in \mathcal{U}$
it follows from equation (\ref{solutionform}) that the solutions of the
linear system (\ref{linearsystem}) satisfy $\phi _{t,u}^{-1}=\phi
_{-t,\Theta _{t}u}.$ Therefore, 
\begin{equation}
\mathcal{A}_{\tau }^{\ast }=\varphi _{-\tau }(\mathcal{A}_{\tau }^{-1}).
\label{sets}
\end{equation}

The next proposition states the main properties of the reachable sets of
linear systems, its proof can be found in \cite{JPh}, Proposition 2.

\begin{proposition}
With the previous notations it holds:

\begin{itemize}
\item[1.] $\tau >0$ $\Rightarrow $ $\mathcal{A}_{\tau }=\mathcal{A}_{\leq 
\text{ }\tau }$

\item[2.] $0\leq \tau _{1}\leq \tau _{2}$ $\Rightarrow $ $\mathcal{A}_{\tau
_{1}}\leq \mathcal{A}_{\tau _{2}}$

\item[3.] $g\in G$ $\Rightarrow $ $\mathcal{A}_{\tau }(g)=\mathcal{A}_{\tau
}\varphi _{\tau }(g)$

\item[4.] $\tau _{1},\tau _{2}\geq 0$ $\Rightarrow $ $\mathcal{A}_{\tau
_{1}+\tau _{2}}=\mathcal{A}_{\tau _{1}}\varphi _{\tau _{1}}(\mathcal{A}%
_{\tau _{2}})=\mathcal{A}_{\tau _{2}}\varphi _{\tau _{2}}(\mathcal{A}_{\tau
_{1}})$
\end{itemize}
\end{proposition}

The next result (Lemma 3.1 of \cite{DaSilva1}) shows that the accessible set 
$\mathcal{A}$ is invariant by right translations of elements whose $\mathcal{%
X}$-orbits are contained in $\mathcal{A}$.

\begin{lemma}
\label{pointinvariance} Let $g\in \mathcal{A}$ and assume that $\varphi
_{t}(g)\in \mathcal{A}$ for any $t\in \mathbb{R}$. Then $\mathcal{A}\cdot
g\subset \mathcal{A}$.
\end{lemma}

\begin{remark}
Using the above lemma it follows that 
\begin{equation*}
\text{if }\varphi _{t}(g)\in \mathrm{cl}(\mathcal{A})\text{ for any }t\in 
\mathbb{R}\Rightarrow \mathrm{cl}(\mathcal{A})\cdot g\subset \mathrm{cl}(%
\mathcal{A}).
\end{equation*}
\end{remark}

In order to extend the controllability results in \cite{DaSilva1} from
solvable groups to more general Lie groups, the authors in \cite{DaAy}
introduced the following notion

\begin{definition}
\label{finitess} Let $G$ be a connected Lie group. We say that the Lie group 
$G$ has finite semisimple center if all semisimple Lie subgroups of $G$ have
finite center.
\end{definition}

For such class of Lie groups we have the following result (Theorem 3.9 of 
\cite{DaAy}).

\begin{proposition}
\label{generalcase} Let $G$ be a connected Lie group with finite semisimple
center. If $\mathcal{A}$ is open, then $G^{+,0}\subset \mathcal{A}$ and $%
G^{-, 0}\subset \mathcal{A}^*$.
\end{proposition}

\begin{remark}
We should notice that a sufficient condition to get $G^{+,0}\subset \mathcal{%
A}$ in Theorem 3.9 of \cite{DaAy} is the fact that 
\begin{equation*}
e\in \mathrm{int}\mathcal{A}_{\tau _{0}},\text{ for some }\tau _{0}>0.
\end{equation*}%
However, by Lemma 4.5.2 of \cite{FCWK} we know that this condition is
equivalent to $\mathcal{A}$ being open.
\end{remark}

\begin{remark}
We should also notice that the condition on the openness of $\mathcal{A}$
implies, in particular, that the system satisfies the Lie algebra rank
condition (see Theorem 3.3 of \cite{VAJT}).
\end{remark}

\section{Control sets of a linear system}

In this section we prove our main result. We start with a proposition which
states properties of the subgroups obtained from the $\mathcal{D}$%
-decomposition.

\begin{proposition}
\label{subgroups} It holds :

\begin{itemize}
\item[1.] $G^{+,0}=G^{+}G^{0}=G^{0}G^{+}$ and $G^{-,0}=G^{-}G^{0}=G^{0}G^{-}$

\item[2.] $G^{+}\cap G^{-}=G^{+,0}\cap G^{-}=G^{-,0}\cap G^{+}=\{e\}$

\item[3.] $G^{+,0}\cap G^{-,0}=G^{0}$

\item[4.] All the above subgroups are closed in $G$

\item[5.] If $G$ is solvable then 
\begin{equation}
G=G^{+,0}G^{-}=G^{-,0}G^{+}.  \label{decomposition}
\end{equation}%
Moreover, the fixed points of $\mathcal{X}$ are in $G^{0}$;
\end{itemize}
\end{proposition}

\begin{proof}
The proof of items 1. to 5. can be found in \cite{DaSilva1} Proposition 2.9.
\end{proof}

\ Next we prove the following

\begin{proposition}
If $\mathcal{D}$ is inner and $G^{0}$ is compact, then $G=G^{0}$.
Furthermore,, the decomposition (\ref{decomposition}) also holds when we
just assume that $G^{0}$ is a compact subgroup.
\end{proposition}

\begin{proof}
Let us assume $\mathcal{D}=\mathrm{ad}(X)$ for some $X\in \mathfrak{g}$.
Certainly $X\in \mathfrak{g}^{0}$ and it holds that 
\begin{equation*}
\varphi _{t}(g)=\mathrm{e}^{-tX}\text{ }g\text{ }\mathrm{e}^{tX},\; \; \; \;%
\mbox{ for any }\,t\in \mathbb{R}\text{ and }g\in G.
\end{equation*}%
Consequently, if $G^{0}$ is compact, the $\mathcal{X}$-orbit 
\begin{equation*}
\mathcal{O}(g)=\{ \varphi _{t}(g),\;t\in \mathbb{R}\}
\end{equation*}%
is bounded for any $g\in G$, since it is contained in the compact set $%
K_{g}=G^{0}g\,G^{0}$. In particular, if $g\in G^{-}$, by the $\mathcal{D}$%
-invariance of $\mathfrak{g}^{-}$ we obtain that $\mathcal{O}(g)\subset
G^{-}\cap K_{g}$ is bounded in $G^{-}$.

Considering that $\varphi _{t}|_{G^{-}}$ is an automorphism of $G^{-}$ if
follows that 
\begin{equation*}
\varrho (\varphi _{t}(g),\varphi _{t}(h))\leq \left \vert \left \vert
(d\varphi _{t}|_{G^{-}})_{e}\right \vert \right \vert \text{ }\varrho
(g,h),\; \; \;g,h\in G^{-}\text{and }t\geq 0,
\end{equation*}%
for any left invariant Riemannian metric on $G^{-}$.

On the other hand, since $(d\varphi _{t})_{e}=\mathrm{e}^{t\mathcal{D}}$ and 
$\mathcal{D}|_{\mathfrak{g}^{-}}$ has only eigenvalues with negative real
part there are $c,\mu >0$ such that 
\begin{equation*}
\left \vert \left \vert (d\varphi _{t}|_{G^{-}})_{e}\right \vert \right
\vert =\left \vert \left \vert \mathrm{e}^{t\mathcal{D}|_{\mathfrak{g}%
^{-}}}\right \vert \right \vert \leq c^{-1}\mathrm{e}^{-\mu t}\text{ for any 
}t\geq 0
\end{equation*}%
implying that 
\begin{equation*}
\varrho (\varphi _{t}(g),\varphi _{t}(h))\leq c^{-1}\text{ }\mathrm{e}^{-\mu
t}\text{ }\varrho (g,h),\; \; \;g,h\in G^{-}\text{ and }t\geq 0.
\end{equation*}%
Consequently 
\begin{equation*}
\varrho (\varphi _{-t}(g),\varphi _{-t}(h))\geq c\text{ }\mathrm{e}^{\mu t}%
\text{ }\varrho (g,h),\; \; \;g,h\in G^{-},\text{ }t\geq 0
\end{equation*}%
which shows that $\mathcal{O}(g)$ is bounded in $G^{-}$ if and only, if $g=e$%
. Therefore, if $G^{0}$ is compact we must have $G^{-}=\{e\}$. Analogously $%
G^{+}=\{e\}$ and $G=G^{0}$ as stated.

Assume now that $G^{0}$ is a compact subgroup of $G$ and $\mathcal{D}$ is an
arbitrary derivation. Let $\mathfrak{r=r(g)}$ stands for the solvable
radical of $\mathfrak{g}$ and $R\subset G$ its associated connected solvable
Lie subgroup. Since $\mathfrak{r}$ is a $\mathcal{D}$-invariant ideal of $%
\mathfrak{g}$, we obtain a well induced linear vector field on the
semisimple Lie group $G/R$. 

Now, any derivation on a semisimple Lie algebra is inner. On the other hand,
Lemma 2.3 of \cite{DaSilva1} give us that 
\begin{equation*}
\left( G/R\right) ^{0}=\pi \left( G^{0}\right) \text{ where }\pi
:G\rightarrow G/R
\end{equation*}%
is the canonical projection. Therefore, from the compacity of $G^{0}$ that
we are assuming we get 
\begin{equation*}
G/R=(G/R)^{0}=\pi \left( G^{0}\right) .
\end{equation*}%
Consequently, $G=G^{0}R.$ But $R$ is $\varphi $-invariant, then item 5 of
the previous proposition shows that 
\begin{equation*}
R=R^{+,0}\text{ }R^{-}=R^{-,0}\text{ }R^{+},
\end{equation*}%
where $R^{+},R^{-},R^{0}$ are the connected Lie subgroup of $R$ with Lie
algebras 
\begin{equation*}
\mathfrak{r}\cap \mathfrak{g}^{+},\text{ }\mathfrak{r}\cap \mathfrak{g}^{-},%
\text{ }\mathfrak{r}\cap \mathfrak{g}^{0}
\end{equation*}%
respectively.

Now, by using item 1 of the same proposition we conclude that 
\begin{equation*}
G=G^{+,0}\text{ }G^{-}=G^{-,0}\text{ }G^{+}
\end{equation*}%
as stated.
\end{proof}

\begin{remark}
The condition on the compacity of $G^{0}$ can be weakened. In fact, it is
enough to ask that the connected components of the Lie subgroup of $G/R$
given by the singularities of the induced linear vector field on $G/R$ are
compacts.
\end{remark}

Next we show that if $G$ has the decomposition (\ref{decomposition}), the
reachable and controllable sets have also such a decompositions.

\begin{lemma}
\label{decompositionreachable} Let us assume that $G$ has finite semisimple
center and$\mathcal{A}$ is open. If $G$ has the decomposition (\ref%
{decomposition}) then 
\begin{equation}
\mathcal{A}=\mathcal{A}_{G^{-}}G^{+,0}\; \; \; \mbox{ and }\; \; \; \mathcal{A}%
^{\ast }=\mathcal{A}_{G^{+}}^{\ast }G^{-,0},
\end{equation}%
where 
\begin{equation*}
\mathcal{A}_{G^{-}}=\mathcal{A}\cap G^{-}\text{and }\mathcal{A}%
_{G^{+}}^{\ast }=\mathcal{A}^{\ast }\cap G^{+}.
\end{equation*}
\end{lemma}

\begin{proof}
We only show the decomposition for $\mathcal{A}$ since the other case is
analogous. By hypothesis, $G^{+,0}\subset \mathcal{A}$. The $\varphi $%
-invariance of $G^{+,0}$ implies by Lemma \ref{pointinvariance} that 
\begin{equation*}
\mathcal{A}_{G^{-}}G^{+,0}\subset \mathcal{A}G^{+,0}\subset \mathcal{A}.
\end{equation*}%
Reciprocally, let $x\in \mathcal{A}$. By decomposition (\ref{decomposition})
there are $a\in G^{-}$ and $b\in G^{+,0}$ such that $x=ab$. Moreover, 
\begin{equation*}
a=xb^{-1}\in \mathcal{A}b^{-1}\subset \mathcal{A}\Rightarrow a\in \mathcal{A}%
_{G^{-}}.
\end{equation*}%
Therefore, $\mathcal{A}=\mathcal{A}_{G^{-}}G^{+,0}$ as desired.
\end{proof}

From now we assume that $G$ has finite semisimple center. We notice that
there are many Lie groups satisfying this property. Of course, any solvable
group and any semisimple Lie group with finite center, like $\mathfrak{sl(}n,%
\mathbb{R}\mathfrak{),}$ has the finite semisimple center property. But
also, the direct or semidirect product between groups with finite semisimple
center have the same property. Furthermore, according to the Levi
Decomposition Theorem given a Lie algebra $\mathfrak{g}$ there exists a
semisimple subalgebra $\mathfrak{s}$ of $\mathfrak{g}$ such that 
\begin{equation*}
\mathfrak{g=r(g)\oplus s}
\end{equation*}%
where $\mathfrak{r(g)}$ is the solvable radical, i.,e., the largest solvable
subgroup of $G$.

In particular, if one Levi subgroup of $G$ has finite center, then $G$ has
finite semisimple center (see Theorem 4.3 of \cite{ALEB} and corollaries
therein).

Let us also assume from now that the reachable set $\mathcal{A}$ is open.
Since $0\in \mathrm{int}\Omega $, Corollary 4.5.11 of \cite{FCWK} assures
the existence of a control set $\mathcal{C}$ of the linear system (\ref%
{linearsystem}) that contains the identity element $e\in G$ in its interior.
Our aim here is to analyze the topological properties of this control set
and to understand in which cases it is in fact the only control set of (\ref%
{linearsystem}) with nonempty interior. A first result in this direction is
the following

\begin{theorem}
Assume that $G$ has finite semisimple center and the reachable set $\mathcal{%
A}$ of the linear system $\Sigma $ is open. For the existent control set $%
\mathcal{C}$ it holds that

\begin{itemize}
\item[1.] $\mathcal{C}$ is closed if and only if $\mathcal{A}^{\ast }=G$

\item[2.] $\mathcal{C}$ is open if and only if $\mathcal{A}=G$

\item[3.] Furthermore, if $G$ is nilpotent we have

\begin{itemize}
\item[i)] $\mathcal{C}$ is closed if and only if $\mathcal{D}$ has only
eigenvalues with nonpositive real part

\item[ii)] $\mathcal{C}$ is open if and only if $\mathcal{D}$ has only
eigenvalues with nonnegative real part

\item[iii)] $\mathcal{C}=G$ if and only if $\mathcal{D}$ has only
eigenvalues with zero real part
\end{itemize}
\end{itemize}
\end{theorem}

\begin{proof}
1. If $\mathcal{A}^{\ast }=G$ we get 
\begin{equation*}
\mathcal{C}=\mathrm{cl}(\mathcal{A})\cap \mathcal{A}^{\ast }=\mathrm{cl}(%
\mathcal{A})
\end{equation*}%
showing that $\mathcal{C}$ is closed. Reciprocally, if $\mathcal{C}$ is
closed, by item 4. of Proposition \ref{controlsets} it follows that $%
\mathcal{C}=\mathrm{cl}(\mathcal{A}(g))$ for any $g\in G$ and in particular $%
\mathcal{C}=\mathrm{cl}(\mathcal{A})$.

By item 2. of the same Proposition we must have $\mathcal{A}\subset \mathcal{%
A}^{\ast }$. Since $G^{+}\subset \mathcal{A}$ Lemma \ref{pointinvariance}
shows that 
\begin{equation*}
B=G^{+}G^{-,0}\subset \mathcal{A}^{\ast }G^{-,0}\subset \mathcal{A}^{\ast }.
\end{equation*}%
Since $B$ is $\varphi $-invariant, Lemma \ref{pointinvariance} implies also
that $B^{n}\subset \mathcal{A}^{\ast }$. Moreover, $B$ is a neighborhood of $%
e\in G$ and since $G$ is connected we get 
\begin{equation*}
G=\bigcup_{n\in \mathbb{N}}B^{n}\subset \mathcal{A}^{\ast }\Rightarrow 
\mathcal{A}^{\ast }=G.
\end{equation*}

2. If $\mathcal{A}=G$ we obtain 
\begin{equation*}
\mathcal{C}=\mathrm{cl}(\mathcal{A})\cap \mathcal{A}^{\ast }=\mathcal{A}%
^{\ast }
\end{equation*}%
showing in particular that $\mathcal{C}$ is open. Conversely, let us assume
that $\mathcal{C}=\mathcal{A}^{\ast }(g)$ for any $g\in \mathcal{C}$. By
item 2. of Proposition \ref{controlsets} we have that $\mathrm{int}\,%
\mathcal{C}\subset \mathcal{A}(g)$ for any $g\in \mathcal{C}$. In
particular, since $\mathcal{C}=\mathcal{A}^{\ast }$ and $\mathcal{A}^{\ast }$
is open we get $\mathcal{A}^{\ast }\subset \mathcal{A}$. Moreover, since $%
G^{-}\subset \mathcal{A}^{\ast }$, $G^{+,0}\subset \mathcal{A}$ and both are 
$\varphi $-invariant, we obtain 
\begin{equation*}
B=G^{+}G^{-,0}=G^{+,0}G^{-}\subset \mathcal{A}G^{-}\subset \mathcal{A}.
\end{equation*}%
As before, this fact allows to us to conclude that $\mathcal{A}=G$ as
desired.

3. If $G$ is nilpotent, Proposition 4.4 of \cite{DaSilva1} assures that $%
\mathcal{A}^*=G$ if and only if $G=G^{-, 0}$ and $\mathcal{A}=G$ if and only
if $G=G^{+, 0}$ which proves item i) and ii). Item iii) follows from Theorem
4.5 of \cite{DaSilva1}.
\end{proof}

The above theorem implies in particular that if $\mathcal{C}$ is an
invariant control set, it is the only invariant one. In fact,

\begin{corollary}
Assume that $G$ has finite semisimple center and the reachable set $\mathcal{%
A}$ of the linear system $\Sigma $ is open. If the existent control set $%
\mathcal{C}$ is invariant, then it is the only invariant one.
\end{corollary}

\begin{proof}
Let us prove the case where $\mathcal{C}$ is invariant in positive time. By
item 4. of Proposition \ref{controlsets} that situation happens if and only
if $\mathcal{C}$ is closed and by the theorem above, if and only if $%
\mathcal{A}^{\ast }=G$ .

If $\widetilde{\mathcal{C}}$ is another closed control set, for any $g\in 
\widetilde{\mathcal{C}}\subset \mathcal{A}^{\ast }$ there are $\tau >0,u\in 
\mathcal{U}$ such that $\phi _{\tau ,u}(g)=e$. Moreover, by the invariance
of $\widetilde{\mathcal{C}}$ we must have 
\begin{equation*}
e=\phi _{\tau ,u}(g)\in \phi _{\tau ,u}(\widetilde{\mathcal{C}})\subset 
\widetilde{\mathcal{C}}.
\end{equation*}%
Showing that $e\in \mathcal{C}\cap \widetilde{\mathcal{C}}$.

Since two control sets with nonempty intersection must coincide we must have 
$\mathcal{C}=\widetilde{\mathcal{C}}.$ Then, $\mathcal{C}$ is the only
invariant one.
\end{proof}

\begin{remark}
We should notice that the above results show us how restrictive is the
condition for a linear system to have an invariant control set. A linear
system admits a positive/negative invariant control set if and only if the
whole group $G$ is controllable to/reachable from the identity. Furthermore,
in \cite{JPh} the author shows that under the LARC condition, the positive
orbit $\mathcal{A}$ of an unrestricted linear control systems is a semigroup
if and only if $\mathcal{A}$ $=G.$ Here, unrestricted means (\ref%
{linearsystem}) with $\Omega =\mathbb{R}^{m}.$
\end{remark}

\subsection{Bounded control sets}

In this section our interest is to search for conditions when our control
set $\mathcal{C}$ is bounded. As before we are assuming that $G$ is a
connected Lie group with finite semisimple center and that the linear system
on $G$ is such that the reachable set $\mathcal{A}$ is open.

Let us consider as before the sets 
\begin{equation*}
\mathcal{A}_{G^{-}}=\mathcal{A}\cap G^{-}\text{ and }\mathcal{A}%
_{G^{+}}^{\ast }=\mathcal{A}^{\ast }\cap G^{+}
\end{equation*}%
Since $\mathcal{A}_{G^{-}}$, $\mathcal{A}_{G^{+}}^{\ast }$ and $G^{0}$ are
contained in $\mathcal{A}\cap \mathcal{A}^{\ast }$, if the control set $%
\mathcal{C}$ is bounded we get 
\begin{equation*}
\mathrm{cl}(\mathcal{A}_{G^{-}}),\; \; \mathrm{cl}(\mathcal{A}_{G^{+}}^{\ast
})\; \; \mbox{ and }\; \;G^{0}\mbox{ are compact sets}.
\end{equation*}

Next we show that in some cases the compacity of such sets imply that $%
\mathcal{C}$ is bounded.

\begin{theorem}
\label{bounded2} Let us assume that $G$ is semisimple or nilpotent. If $%
\mathrm{cl}(\mathcal{A}_{G^{-}})$, $\mathrm{cl}(\mathcal{A}_{G^{+}}^{\ast })$
and $G^{0}$ are compact subsets of $G$ then $\mathcal{C}$ is bounded.
\end{theorem}

\begin{proof}
If $G$ is semisimple, the result follows direct from Proposition 3.2 since
in this case $G=G^{0}$. For the nilpotent case we prove by induction on the
dimension of $G$.

If $\dim G=1$ then $G$ is Abelian and by Lemma \ref{decompositionreachable}
we have that 
\begin{equation*}
\mathcal{A}\cap \mathcal{A}^{\ast }=\mathcal{A}_{G^{-}}G^{-,0}\cap \mathcal{A%
}_{G^{+}}^{\ast }G^{+,0}=\mathcal{A}_{G^{-}}G^{0}\mathcal{A}_{G^{+}}^{\ast }.
\end{equation*}%
Considering that $\mathcal{A}\cap \mathcal{A}^{\ast }$ is dense in $\mathcal{%
C}$ we get 
\begin{equation*}
\mathcal{C}\subset \mathrm{cl}(\mathcal{A}_{G^{-}})G^{0}\mathrm{cl}(\mathcal{%
A}_{G^{+}}^{\ast })
\end{equation*}%
which by our assumptions is a compact set, implying that $\mathcal{C}$ is a
bounded control set.

Let us assume that $G$ is a nilpotent Lie group with dimension $n$.

Let $Z_G$ be the center of $G$ and $(Z_G)_{0}$ its connected component of
the identity. Since $(Z_G)_{0}$ is $\varphi $-invariant and Abelian, we have
by item 5. of Proposition \ref{subgroups} that $(Z_G)_{0}=Z^{+}Z^{0}Z^{-}$.
Moreover $Z^{+}$, $Z^{0}$ and $Z^{-}$ are $\varphi $-invariant normal
subgroups of $G$. Since $G$ is nilpotent, $(Z_G)_{0}$ is nontrivial and
consequently at least one of the subgroups $Z^{+}$, $Z^{0}$ or $Z^{-}$ is
nontrivial.

Let us analyze the case when $\left \{ e\right \} \varsubsetneq Z^{+}$, the
other cases are analogous. Consider the connected nilpotent Lie group $%
H=G/Z^{+}$. By the $\varphi $-invariance of $Z^{+}$ we get an induced linear
system on $H$ (see Proposition 4 of \cite{JPh1}) that satisfy 
\begin{equation*}
\pi (\phi _{t,u}^{G}(g))=\phi _{t,u}^{H}(\pi (g))
\end{equation*}%
where $\pi :G\rightarrow H$ is the canonical projection.

The above equation gives us that $\pi (\mathcal{A})$ and $\pi (\mathcal{A}%
^{\ast })$ are the reachable and controllable sets of the identity in $H$
respectively. Considering that 
\begin{equation*}
\dim H=\dim G-\dim Z^{+}<n
\end{equation*}%
by hypothesis it follows that 
\begin{equation*}
\mathcal{C}^{H}=\mathrm{cl}(\pi (\mathcal{A}))\cap \pi (\mathcal{A}^{\ast })%
\text{ is a bounded control set.}
\end{equation*}%
Since $\pi $ is an open map, there exists a compact set $K\subset G$ such
that 
\begin{equation*}
\pi (\mathcal{C})\subset \mathcal{C}^{H}\subset \pi (K)\text{ implying that }%
\mathcal{C}\subset KZ^{+}.
\end{equation*}%
Since $G^{0}$ is compact, $G$ has the decomposition (\ref{decomposition})
and we can assume without lost of generality that 
\begin{equation*}
K=K^{+}K^{-,0},\text{ }K^{+}\text{and }K^{-,0}\text{ are compact subsets of }%
G^{+}\text{and }G^{-,0}
\end{equation*}%
respectively.

Take $x\in \mathcal{A}\cap \mathcal{A}^{\ast }.$ There are 
\begin{equation*}
k_{1}\in K^{+},k_{2}\in K^{-,0}\text{ and }z\in Z^{+};\;x=k_{1}k_{2}z.
\end{equation*}%
But $Z^{+}\subset (Z_{G})_{0}$ so we obtain $x=(zk_{1})k_{2}$. Since by
Lemma \ref{decompositionreachable} we have $\mathcal{A}^{\ast }=\mathcal{A}%
_{G^{+}}^{\ast }G^{-,0}$ and since $x\in \mathcal{A}^{\ast }$ we get that 
\begin{equation*}
zk_{1}\in \mathcal{A}_{G^{+}}^{\ast }\text{ or equivalently \ }z\in \mathcal{%
A}_{G^{+}}^{\ast }(K^{+})^{-1}.
\end{equation*}%
We started with an arbitrary element $x\in \mathcal{A}\cap \mathcal{A}^{\ast
},$ then we can conclude that 
\begin{equation*}
\mathcal{A}\cap \mathcal{A}^{\ast }\subset L=K\mathrm{cl}(\mathcal{A}%
_{G^{+}}^{\ast })(K^{+})^{-1}.
\end{equation*}%
By hypothesis $\mathrm{cl}(\mathcal{A}_{G^{+}}^{\ast })$ is compact and $%
\mathcal{A}\cap \mathcal{A}^{\ast }$ is dense in $\mathcal{C},$ thus $%
\mathcal{C}$ is contained in the compact set $L$, finishing the proof.
\end{proof}

For simply connected nilpotent Lie groups we have the following:

\begin{corollary}
Let $G$ be a nilpotent simply connected Lie group. Then $\mathcal{C}$ is
bounded if, and only if, $\mathrm{cl}(\mathcal{A}_{G^{-}})$ and $\mathrm{cl}(%
\mathcal{A}_{G^{+}}^{\ast })$ are compact subsets of $G$ and $\mathcal{D}$
is hyperbolic.
\end{corollary}

\begin{proof}
In fact, by Theorem \ref{bounded2} $\mathcal{C}$ is bounded if and only if $%
G^0$, $\mathrm{cl}(\mathcal{A}_{G^{-}})$ and $\mathrm{cl}(\mathcal{A}%
_{G^{+}}^{\ast })$ are compact subsets. Since $G$ is simply connected its
exponential $\exp $ is a diffeomorphism which implies that $G^{0}$ is
compact if and only if $\mathfrak{g}^{0}$ is compact if and only if $%
\mathfrak{g}^{0}=\{0\}$.
\end{proof}

\subsection{Uniqueness}

In the classical Euclidean Abelian Linear Control Systems it is well known
that under the Kalman rank condition there exists just one control set with
non empty interior. We will show here that any solvable Lie group and any
group $G$ such that $G^{0}$ is a compact subgroup have the same property, i.
e., they have at most one control set. We start with,

\begin{theorem}
\label{unicity} The set $\mathcal{C}$ is the only control set of the linear
system (\ref{linearsystem}) whose interior intersects $G^{+,0}G^{-}$ and $%
G^{-,0}G^{+}$.
\end{theorem}

\begin{proof}
Let $\widetilde{\mathcal{C}}$ be a control set such that its interior
intersects $G^{+, 0}G^-$ and $G^{-, 0}G^+$. Since two control sets with
nonempty intersection must coincide it is enough to show that $\mathcal{C}%
\cap \widetilde{\mathcal{C}}\neq \emptyset$. We will divide the proof in
three steps:

\textit{Step 1: If $\mathrm{int}\, \widetilde{\mathcal{C}}\cap
G^{+,0}G^{-}\neq \emptyset $ there are $\tau _{1}>0$ and $u_{1}\in \mathcal{U%
}$ such that 
\begin{equation}
a=\phi _{\tau _{1},u_{1}}\in \mathrm{int}\, \widetilde{\mathcal{C}};
\label{1}
\end{equation}%
} Let $x\in \mathrm{int}\, \widetilde{\mathcal{C}}\cap G^{+,0}G^{-}$ and let $%
g\in G^{+,0}$ and $h\in G^{-}$ such that $x=gh$. Considering that in $%
\mathrm{int}\, \widetilde{\mathcal{C}}$ controllability holds, there are $%
\tau >0$ and $u\in \mathcal{U}$ such that $\phi _{\tau ,u}(x)=x$. Let $%
\varrho $ be a left invariant Riemannian metric on $G$. Since $\varphi
_{t}(h)\rightarrow e$ as $t\rightarrow +\infty $ we have that 
\begin{equation*}
\varrho (\phi _{t,u}(x),\phi _{t,u}(g))=\varrho (\phi _{t,u}(g)\varphi
_{t}(h),\phi _{t,u}(g))=\varrho (\varphi _{t}(h),e)\overset{t\rightarrow 
\text{ }+\infty }{\rightarrow }0.
\end{equation*}%
Since $\phi _{n\tau ,u}(x)=x$, for any $n\in \mathbb{N}$, we obtain that $%
\phi _{t,u}(g)\in \mathrm{int}\, \widetilde{\mathcal{C}},$ for $t>0$ greater
enough. Moreover, $g\in G^{+,0}\subset \mathcal{A}$ which implies, by
concatenation, that 
\begin{equation*}
\phi _{\tau _{1},u_{1}}\in \mathrm{int}\, \widetilde{\mathcal{C}}\; \; \;%
\mbox{
for some }\; \; \; \tau _{1}>0,u_{1}\in \mathcal{U},
\end{equation*}%
as stated.

\textit{Step 2: If $\mathrm{int}\, \widetilde{\mathcal{C}}\cap G^{-,
0}G^+\neq \emptyset$ there are $\tau_2>0$ and $u_2\in \mathcal{U}$ such that 
\begin{equation}  \label{2}
e\in \phi_{\tau_2, u_2}(\mathrm{int}\, \widetilde{\mathcal{C}});
\end{equation}
}

Let $x^{\prime }\in \mathrm{int}\, \widetilde{\mathcal{C}}\cap G^{-,0}G^{+}$
and $g^{\prime -,0}$, $h^{\prime +}$ such that $x^{\prime }=g^{\prime
}h^{\prime }$. Again by controllability in $\mathrm{int}\, \widetilde{%
\mathcal{C}}$ there are $\tau ^{\prime }>0$, $u^{\prime }\in \mathcal{U}$
such that $\phi _{-\tau ^{\prime },u^{\prime }}(x^{\prime })=x^{\prime }$.
Since $\varphi _{-t}(h^{\prime })\rightarrow e$ as $t\rightarrow +\infty $
we have 
\begin{equation*}
\varrho (\phi _{-t,u^{\prime }}(x^{\prime }),\phi _{-t,u^{\prime
}}(g^{\prime }))=\varrho (\phi _{-t,u^{\prime }}(g^{\prime })\varphi
_{t}(h^{\prime }),\phi _{-t,u}(g^{\prime }))=\varrho (\varphi
_{-t}(h^{\prime }),e)\overset{t\rightarrow \text{ }+\infty }{\rightarrow }0.
\end{equation*}%
Again, since $\varphi _{-n\tau ^{\prime },u^{\prime }}(x^{\prime
})=x^{\prime }$ for any $n\in \mathbb{N}$ we get that $\phi _{-t,u^{\prime
}}(g^{\prime })\in \mathrm{int}\, \widetilde{\mathcal{C}},$ for $t>0$ greater
enough. Moreover, $g^{\prime -,0}\subset \mathcal{A}^{\ast }$ which gives us
by concatenation that for some $\tau _{2}>0$, $\Theta _{\tau _{2}}u_{2}\in 
\mathcal{U}$ we have 
\begin{equation*}
\phi _{-\tau _{2},\Theta _{\tau _{2}}u_{2}}\in \mathrm{int}\, \widetilde{%
\mathcal{C}}\; \; \; \Leftrightarrow \; \; \;e\in \phi _{\tau _{2},u_{2}}(\mathrm{%
int}\, \widetilde{\mathcal{C}}),
\end{equation*}

\textit{Step 3: It holds that $\mathcal{C}\cap \widetilde{\mathcal{C}}\neq
\emptyset$};

By (\ref{2}) there is $b\in \mathrm{int}\, \widetilde{\mathcal{C}}$ such
that $\phi_{\tau_2, u_2}(b)=e$. Let $a=\phi_{\tau_1, u_1}\in \mathrm{int}\,%
\widetilde{\mathcal{C}}$ given in (\ref{1}). By controllability in $\mathrm{%
int}\, \widetilde{\mathcal{C}}$ there are $\tau_3>0$ and $u_3\in \mathcal{U}$
such that $\phi_{\tau_3, u_3}(a)=b$. If we concatenate $u_1, u_2$ and $u_3$
and extend it periodically we obtain an admissible control function $u\in%
\mathcal{U}$ with period $T=\tau_1+\tau_2+\tau_3$. Moreover, $\phi_{t, u}(e)$
is a periodic solution of the linear system (\ref{linearsystem}) passing by $%
a, b\in \mathrm{int}\, \widetilde{\mathcal{C}}$. Since periodic orbits
cannot leave the interior of a control set we must have that it lies in $%
\mathcal{C}\cap \widetilde{\mathcal{C}}$ which concludes the proof.
\end{proof}

As a direct corollary we have:

\begin{corollary}
If $G$ is a solvable Lie group with compact $G^{0}$-component, then $%
\mathcal{C}$ is the only control set.
\end{corollary}

\begin{proof}
If $G=G^{+,0}G^{-}=G^{-,0}G^{+}$ it follows that any control set of (\ref%
{linearsystem}) with nonempty interior is such that its interior intersects $%
G^{+,0}G^{-}$ and $G^{-,0}G^{+}$. Consequently, by Proposition \ref%
{subgroups} item 5., if $G$ is a solvable Lie group or if $G^{0}$ is a
compact subgroup, then $\mathcal{C}$ is the only control set of the linear
system.
\end{proof}


\end{document}